\documentclass [11pt,oneside,a4paper,mathscr]{amsart}

\usepackage{amscd,amsmath,amssymb,euscript}
\usepackage[frame,cmtip,curve,arrow,matrix,line,graph]{xy}

\title[]{Solutions of the Strominger System via Stable Bundles on Calabi-Yau Threefolds}
\author{Bj\"orn Andreas and Mario Garcia-Fernandez}
\address{Institut f\"ur Mathematik, Humboldt-Universit\"at zu Berlin, Rudower Chaussee 25, 12489 Berlin, Germany.}
\email{andreas\char`\@math.hu-berlin.de}
\address{Centre for Quantum Geometry of Moduli Spaces, Aarhus University, Ny Munkegade 118, bldg. 1530, DK-8000 Aarhus C, Denmark.}
\email{mariogf\char`\@qgm.au.dk}

\date{}
\newtheorem{thm}{Theorem}[section]

\newtheorem{prop}[thm]{Proposition}
\newtheorem{lemma}[thm]{Lemma}

\theoremstyle{definition}

\newtheorem{thm*}[thm]{Theorem$^*$}

\DeclareMathOperator{\tr}{tr}

\def\dbar{\bar\partial}

\begin{document}

\begin{abstract}
We prove that a given Calabi-Yau threefold with a stable holomorphic vector bundle
can be perturbed to a solution of the Strominger system provided that the second Chern
class of the vector bundle is equal to the second Chern class of the tangent bundle.
If the Calabi-Yau threefold has strict $SU(3)$ holonomy then the equations of motion derived from the
heterotic string effective action are also satisfied by the solutions we obtain.

\end{abstract}
\maketitle
\section{Introduction}\label{sec:intro}
In a series of papers Fu, Li, Tseng and Yau \cite{Yau1, FuYau, FTY, LiYau} have provided first examples of solutions
of a system of coupled non-linear differential equations which arise as consistency conditions
in heterotic string compactifications. The system of equations is called the Strominger system and has been originally proposed by Strominger in \cite{Strom}. From a mathematical perspective, it can be considered as a generalization of the K\"ahler Ricci-flat equation for the case of non-K\"ahler Calabi-Yau manifolds~\cite{Yau1}. Moreover, the Strominger system is expected \cite{Yau1, Adams} to play an important role in investigating the geometry of Calabi-Yau threefolds within the Reid conjecture
\cite{Reid} which relates all moduli spaces of smooth Calabi-Yau threefolds by certain
birational transformations.

To solve the Strominger system one has to specify a conformally balanced Hermitian form
on a compact complex three-dimensional manifold $X$, a nowhere vanishing holomorphic $(3,0)$
form and a Hermitian-Yang-Mills connection on a bundle $E$ over this manifold. The consistency
of the underlying physical theory imposes a constraint on the associated curvature forms of
the connection on the bundle $E$ and of a unitary connection on the tangent bundle of $X$. More precisely, the curvature
forms have to satisfy the anomaly equation, a non-linear partial differential equation. The necessary condition
for the existence of solutions of this equation is that the second Chern classes of $E$ and $X$ agree.
However, although this topological constraint can be solved within the framework of algebraic geometry by finding suitable stable bundles, the problem which remains is to prove that solutions of the Strominger system actually exist.

One difficulty in obtaining smooth solutions of the Strominger system lies in the fact that many theorems
of K\"ahler geometry and thus methods of algebraic geometry do not apply. Therefore one approach to
obtain solutions is to simultaneously perturb K\"ahler and Hermitian-Yang-Mills metrics and so avoid the direct construction of non-K\"ahler manifolds with stable bundles. This approach has been used in \cite{LiYau} where it is shown that a deformation of the holomorphic structure on the direct sum of the tangent bundle and the trivial bundle and of the K\"ahler form of a given Calabi-Yau threefold leads to a smooth solution of the Strominger system whereas the original Calabi-Yau space is perturbed to a non-K\"ahler space. One question which motivated the present paper was, whether this method can be used also for any stable bundle (which satisfy the topological constraint imposed by the anomaly equation). A result along this lines has been
originally conjectured in a slightly different framework by Witten \cite{Wi86} who gave evidence for this in \cite{WiWi87,WuWi87}.

In this paper we will obtain solutions of the Strominger system using a perturbative method which has been inspired by the method developed in \cite{LiYau}. Our starting point will be a solution of the Strominger system in the large radius limit, given by any stable vector bundle on a Calabi-Yau threefold which satisfies the second Chern class constraint. This solution is then perturbed to a solution of the system using the implicit function theorem. Moreover, if the initial Calabi-Yau threefold has
strict $SU(3)$ holonomy, the obtained solutions satisfy also the equations of motion derived from the effective action of the heterotic string. For this we use a result of \cite{FIVU} and extend the original Strominger system by an instanton condition. For the proof of our main result (see Theorem~\ref{thm:existence0}) we rely on a previous result in~\cite{LiYau}. Previous examples of simultaneous solutions of the Strominger system and the equations of motion on compact nilmanifolds have been obtained in \cite{FIVU}.
\vskip .5cm
This paper is organized as follows. In Section 2 we briefly review for convenience various aspects of the Strominger system and refer for more details to \cite{Strom, CHSW}. In Section 3 we explain the main result of the paper and the method we use for its proof. In Section 4 we prove Theorem~\ref{thm:existence0}. In Section 5 we give some examples which satisfy the second Chern class constraint and so the hypothesis of Theorem~\ref{thm:existence0}. Finally, in Section 6 we discuss some further directions for research.
\vskip .5cm
{\bf Acknowledgments.} We thank A. Prat-Waldron, L. \'Alvarez-C\'onsul, G. Curio, D. Hern\'andez-Ruip\'erez, E. Witten and S.-T. Yau for discussion. We thank the SFB 647 and the Humboldt University Berlin for financial support and the Free University Berlin for hospitality. M.G.F. thanks the Max-Planck-Institute for Mathematics for financial support where part of this work was carried out.

\section{Preliminaries}
In \cite{Strom} Strominger proposed to specify a ten-dimensional space-time that is a (warped) product of a maximal symmetric four-dimensional space-time and a compact complex three-dimensional manifold $X$ in order to compactify the heterotic string. On $X$ one has to specify a Hermitian form $\omega$, a holomorphic volume form $\Omega$ which is a nowhere vanishing section of the canonical bundle $K_X={\mathcal O}_X$ and a connection $A$ on a complex vector bundle $E$ over $X$ with vanishing first Chern class $c_1(E) = 0$. To get a supersymmetric theory the gauge field $A$ has to satisfy the Hermitian-Yang-Mills (HYM) equations
\begin{equation}\label{HYM}
F^{2,0}=F^{0,2}=0, \ \ \  F\wedge \omega^2=0,
\end{equation}
where $F$ denotes the curvature two form of the connection $A$. If $\omega$ is closed, the Donaldson-Uhlenbeck-Yau Theorem~\cite{Don,UY} states that if $E$ is a $\omega$-stable holomorphic vector bundle, then there exists a connection $A$ on $E$ such that the HYM equations are satisfied. A similar result holds for arbitrary Hermitian metrics (i.e. for non closed $\omega$) as proved in \cite{LiYauHYM}.

In addition, supersymmetry requires that the Hermitian form $\omega$ and the holomorphic three form $\Omega$ have to satisfy the dilatino equation
\begin{equation}\label{dilaton}
d^*\omega=i(\dbar - \partial)\log||\Omega||_{\omega},
\end{equation}
where $\dbar$ and $\partial$ denote the Dolbeault operator on $X$ and its conjugate and $d^*$ is the adjoint of the exterior differential $d$ with respect to $\omega$.  In~\cite[Lemma~3.1]{LiYau} it is shown that this last equation is equivalent to
\begin{equation}\label{confbalanced}
d(||\Omega||_\omega \omega^2) = 0,
\end{equation}
where $||\Omega||_\omega$ denotes the point-wise norm of $\Omega$ with respect to $\omega$. Note that~\eqref{confbalanced} implies that $\omega$ is conformally balanced. Moreover, to get a consistent physical theory the Hermitian form $\omega$ and the curvature two form $F$ have to
satisfy the anomaly equation
\begin{equation}\label{bianchi}
i\partial\dbar \omega=\alpha'(\tr (R\wedge R)-\tr(F\wedge F)),
\end{equation}
where $R$ is the curvature of a $\omega$-unitary connection on the tangent bundle $TX$ and $\alpha'$ is the slope parameter in string theory. As mentioned in the introduction, the existence of a solution of \eqref{bianchi} requires that the topological condition
\begin{equation}\label{eq:Anomallycohom0}
c_2(E)=c_2(X)
\end{equation}
is satisfied. The coupled system of differential equations \eqref{HYM}-\eqref{bianchi} defines what is called the Strominger system.

As we are interested in also solving the equation of motion derived from the heterotic string effective action, we apply here a recent result obtained in~\cite{Ivan09,FIVU}. This result states that a solution of the Strominger system (the supersymmetry and anomaly equation) implies a solution of the equation of motion if and only if the connection on $TX$ is an $SU(3)$ instanton, that is, the curvature two form $R$ has to satisfy
\begin{equation}\label{Rinst}
R^{2,0}=R^{0,2}=0, \ \ \ R \wedge \omega^2=0.
\end{equation}
As pointed out in \cite{Ivan09}, equation~\eqref{Rinst} completely determines the choice of $\omega$-unitary connection in the anomaly equation~\eqref{bianchi} up to gauge transformations.

\section{Method and Main Result}
\label{sec:methodresult}

We now explain the method we will adopt in order to obtain simultaneous solutions of the Strominger system \eqref{HYM}-\eqref{bianchi}
and \eqref{Rinst}. We first note that the above system is invariant under rescaling of the Hermitian form $\omega$, except for the anomaly equation. Given a positive real constant $\lambda$, if we change $\omega\to\lambda\omega$ and define $\epsilon:= \alpha'/ \lambda$ we obtain the new system
\begin{align}\label{stromsys}
F^{2,0}=F^{0,2}=0, \ \ \  F\wedge \omega^2&=0,\\
d(||\Omega||_\omega \omega^2) & = 0,\label{eq:epsilonconfbalanced}\\
i\partial\dbar\omega - \epsilon(\tr(R\wedge R)-\tr(F\wedge F))&=0,\label{anom}\\
R^{2,0}=R^{0,2}=0, \ \ \ R\wedge \omega^2&=0,\label{eq:epsiloninstR}
\end{align}
that will be called in the following the \emph{$\epsilon$-system}. Here we use the equivalence between the dilatino equation \eqref{dilaton} and the conformally balanced condition \eqref{confbalanced}, given by \cite[Lemma~3.1]{LiYau}. Therefore, any solution of the $\epsilon$-system with $\epsilon > 0$ is related to the original system after rescaling. In the limit $\lambda\to\infty$ a solution of the $\epsilon$-system is given by a degree zero stable holomorphic vector bundle $E$ on a Calabi-Yau threefold, as we state in Lemma~\ref{lemma:L0}.
The next step is then to perturb a given solution with $\epsilon=0$ to a solution with small $\epsilon > 0$, i.e.,
with large $\lambda$. For this, we perturb the K\"ahler form of the given Calabi-Yau threefold into a conformally balanced Hermitian form on the fixed complex manifold while also perturbing its Chern connection and the unique HYM connection on the bundle $E$, whereas we preserve the HYM condition. Moreover, in order for \eqref{anom} to have a solution, the topological obstruction
\begin{equation}\label{eq:Anomallycohom}
c_2(E) = c_2(X),
\end{equation}
must be satisfied. This provides our starting point.

Let $(X,\omega_0)$ be a compact Calabi-Yau threefold with nowhere vanishing $(3,0)$-holomorphic form $\Omega$. Let $E$ be a stable holomorphic vector bundle over $X$ satisfying \eqref{eq:Anomallycohom} with rank $r$ and $c_1(E) = 0$. In this setting an unknown of the $\epsilon$-system corresponds to a triple $(h,\omega,\dbar_{_{TX}})$ where $h$ is a Hermitian metric on $E$, $\omega$ is a Hermitian form on $X$ and
$$
\dbar_{_{TX}}\colon \Omega^0(TX) \to \Omega^{0,1}(TX),
$$
is an integrable Dolbeault operator on $TX$ (regarded as a smooth complex vector bundle) satisfying $\dbar_{_{TX}}^2 = 0$. Recall that the integrability condition implies that there exist holomorphic coordinates on $TX$ with respect to $\dbar_{_{TX}}$ (see e.g.~\cite{NijWoolf}). The curvatures $R$ and $F$ correspond to the Chern connections of $\omega$ and $h$ on the holomorphic bundles $(TX,\dbar_{_{TX}})$ and $E$, respectively. In holomorphic coordinates for $\dbar_{_{TX}}$ and $E$ we can write
\begin{equation}\label{eq:curvaturesformula}
R = \dbar(h_\omega^{-1}\partial h_\omega) \qquad \textrm{and} \qquad F = \dbar(h^{-1}\partial h),
\end{equation}
where $h_\omega$ denotes the Hermitian metric on $TX$ determined by $\omega$. Note that the condition on the $(2,0)$ and $(0,2)$ part of $R$ and $F$ in~\eqref{eq:epsiloninstR} and~\eqref{stromsys} is always satisfied.

The fixed data $(E,X,\omega_0)$ provide a \emph{canonical solution} $(h_0,\omega_0,\dbar_0)$ in the limit $\lambda \to \infty$, where $h_0$ is the unique Hermite-Einstein metric on $E$, given by the Donaldson-Uhlembeck-Yau Theorem, and $\dbar_0$ is the Dolbeault operator on $TX$ determined by the holomorphic structure on $X$. Our main result is:
\begin{thm}\label{thm:existence0}
Let $E$ be a degree zero holomorphic vector bundle over a compact Calabi-Yau threefold $(X,\omega_0)$. If $c_2(E) = c_2(X)$ and $E$ is stable with respect to $[\omega_0]$, then there exists $\lambda_0 \gg 0$ and a $C^1$ curve
$$
]\lambda_0,+\infty[ \ni \lambda \to (h_\lambda,\omega_\lambda,\dbar_\lambda)
$$
of solutions of the Strominger system such that $\dbar_\lambda$ is isomorphic to $\dbar_0$ for all $\lambda$ and $(h_\lambda,\frac{\omega_\lambda}{\lambda},\dbar_\lambda)$ converges uniformly to the canonical solution $(h_0,\omega_0,\dbar_0)$ when $\lambda \to \infty$. Moreover, if $(X,\omega_0)$ has holonomy equal to $SU(3)$ then $(h_\lambda,\omega_\lambda,\dbar_\lambda)$ solves also \eqref{Rinst} and so it provides a solution of the equations of motion derived from the effective heterotic string action.
\end{thm}

The perturbative process provided by Theorem~\ref{thm:existence0} leaves the holomorphic structure of $E$ unchanged while the one on $TX$ is shifted by a complex gauge transformation and so remains isomorphic to the initial one. Since $E$ admits an irreducible solution of the Hermite-Einstein equations with respect to any element of the family of conformally balanced metrics $\omega_\lambda$, it is $\omega_\lambda$-stable for any $\lambda$. Recall that the stability condition for a holomorphic vector bundle $E$ over $X$ with respect to $\omega_\lambda$ is defined in terms of slopes of coherent subsheaves as in the K\"ahler case where the degree is computed by $c_1(E)\cdot [||\Omega||_\omega \omega^2]$ (see~\cite{LuebTel}). The same holds for $TX$, when $(X,\omega_0)$ has holonomy equal to $SU(3)$.

We now sketch the idea underlying the perturbative process in Theorem~\ref{thm:existence0}. We consider a $1$-parameter family of maps between suitable fixed (Hilbert) spaces
$$
\textbf{L}^\epsilon\colon \mathcal{V}_1 \to \mathcal{V}_2,
$$
indexed by $\epsilon \in \mathbb{R}$ which acts on triples $(h,\omega, \dbar_{_{TX}}) \in \mathcal{V}_1$ and whose zero locus corresponds to solutions of the $\epsilon$-system. 
This way, the equation $\textbf{L}^\epsilon = 0$ with $\epsilon \neq 0$ corresponding to the $\epsilon$-system is related to a simpler one at $\epsilon = 0$ which has a known canonical solution $(h_0,\omega_0, \dbar_0)$. Proving now that the differential
$$
\delta_0 \textbf{L}^0\colon \mathcal{V}_1 \to \mathcal{V}_2,
$$
at the initial solution is an isomorphism, an implicit function theorem argument shows that there is a solution of $\textbf{L}^\epsilon = 0$ nearby $(h_0,\omega_0, \dbar_0) \in \mathcal{V}_1$ in an open neighborhood of $\epsilon = 0 \in \mathbb{R}$.

To give a more precise description of the method, we introduce some notation that we use in the sequel. Let $(h_0,\omega_0, \dbar_0)$ be the canonical solution of the $\epsilon$-system at $\epsilon = 0$ associated to $E$ and $(X,\omega_0)$. Let $\mathcal{H}(E)_1$ be the space of Hermitian metrics $h$ on $E$ whose induced metric on $\Lambda^r E \cong \mathbb{C}_X$ is the constant metric $1$. We identify its elements with determinant one symmetric endomorphisms of $(E,h_0)$ via
\begin{equation}\label{eq:omegahomega}
\langle u,v\rangle_h = \langle h u,v\rangle_{h_0},
\end{equation}
for $u,v$ smooth sections on $E$. Let $\mathcal{H}(X)$ be the cone of positive definite Hermitian forms on $X$, regarded as an open subspace of $\Omega_{\mathbb{R}}^{1,1}(X)$. As in \eqref{eq:omegahomega}, given $\omega \in \mathcal{H}(X)$ we identify the corresponding Hermitian metric $h_\omega$ with a symmetric endomorphism of $(TX,h_{\omega_0})$. Let $\mathcal{D}(TX)$ be the space of integrable Dolbeault operators on $TX$, regarded as a smooth complex vector bundle. Given a real vector bundle $W$ over $X$, we denote by $\Omega^{p,q}_\mathbb{R}(W)$ and $\Omega^m_\mathbb{R}(W)$ the space of smooth real $(p,q)$-type forms and $m$-forms on $X$ with values in $W$. Let $\operatorname{ad}_0E$ and $\operatorname{ad}TX$ be the smooth real vector bundles of traceless Hermitian antisymmetric endomorphisms of $(E,h_0)$ and Hermitian antisymmetric endomorphisms of $(TX,\omega_0)$, respectively. Then, we define the operator
\begin{equation}\label{eq:Lepsilonsmooth0}
\textbf{L}^\epsilon \colon \mathcal{H}(E)_1 \times \mathcal{H}(X) \times \mathcal{D}(TX) \to \Omega^6_\mathbb{R}(\operatorname{ad}_0E)\oplus\Omega_{\mathbb{R}}^{2,2}(X)\oplus\Omega_{\mathbb{R}}^1(X)\oplus \Omega^6_\mathbb{R}(\operatorname{ad}TX),
\end{equation}
as the direct sum $\textbf{L}^\epsilon = \textbf{L}_1\oplus \textbf{L}_2^\epsilon\oplus\textbf{L}_3\oplus\textbf{L}_4$, where
\begin{equation}\label{eq:Lepsilonsplit}
\begin{split}
\textbf{L}_1(h,\omega) & = (h^{1/2} \cdot F \cdot h^{-1/2})\wedge \omega^2  \qquad \qquad \;\;\;\;\;\in  \Omega^6_\mathbb{R}(\operatorname{ad}_0E),\\
\textbf{L}_2^\epsilon(h,\omega,\dbar_{_{TX}}) & = i\partial\dbar \omega - \epsilon(\tr R\wedge R -\tr F\wedge F)  \qquad \in \Omega_{\mathbb{R}}^{2,2}(X), \\
\textbf{L}_3(\omega) & = *d(||\Omega||_\omega\omega^2) \qquad \qquad \qquad \qquad \;\;\;\;\; \in \Omega_{\mathbb{R}}^1(X),\\
\textbf{L}_4(\omega,\dbar_{_{TX}}) & = (h_\omega^{1/2}\cdot R \cdot h_\omega^{-1/2})\wedge\omega^2 \qquad \qquad \;\;\;\;\; \in \Omega^6_\mathbb{R}(\operatorname{ad}TX).
\end{split}
\end{equation}
Here $*$ is the Hodge star operator of $\omega_0$, $R = R_{\omega,\dbar_{_{TX}}}$ is the curvature of the Chern connection of $\omega$ with respect to $\dbar_{_{TX}}$ and $F = F_h$ is the curvature of the Chern connection of $h$ on the holomorphic bundle $E$. Since $R$ and $F$ are of type $(1,1)$ (see~\eqref{eq:curvaturesformula}), we have $h^{1/2}F h^{-1/2}\in \Omega_\mathbb{R}^{1,1}(\operatorname{ad}_0E)$ and $h_\omega^{1/2}R h_{\omega}^{-1/2}\in \Omega_\mathbb{R}^{1,1}(\operatorname{ad}TX)$. Hence, the zero locus of $\textbf{L}^\epsilon$ corresponds to solutions of the $\epsilon$-system and (see Lemma~\ref{lemma:L0})
\begin{equation}\label{eq:L0equals0}
\textbf{L}^0(h_0,\omega_0,\dbar_0) = 0.
\end{equation}
We then apply the implicit function theorem to perturb the given solution $(h_0,\omega_0,\dbar_0)$ of \eqref{eq:L0equals0} to a solution of the $\epsilon$-system with $\epsilon > 0$. In order to do this, we consider suitable Hilbert spaces $\mathcal{V}_1$ and $\mathcal{V}_2$ (see~\eqref{eq:V1} and~\eqref{eq:V2}) and $\textbf{L}^\epsilon$ as a map
\begin{equation}\label{eq:Lepsilon0}
\textbf{L}^\epsilon\colon \mathcal{U} \to \mathcal{V}_2,
\end{equation}
where $ \mathcal{U}\subset\mathcal{V}_1$ is an open subset of $\mathcal{V}_1$. The space $\mathcal{V}_1$ is a Sobolev completion of the space of those triples $(h,\mu,\dbar_{_{TX}}) \in \mathcal{H}(E)_1\times \Omega_\mathbb{R}^{1,1}(X) \times\mathcal{D}(TX)$ such that $\dbar_{_{TX}}$ is isomorphic to $\dbar_0$ via a complex gauge transformation on $TX$, while $\mathcal{U}$ is defined by the locus where $\mu$ is positive definite. The space $\mathcal{V}_2$ is simply a Sobolev completion of the codomain in \eqref{eq:Lepsilonsmooth0}. For simplicity, in this section we identify $(h_0,\omega_0,\dbar_0)$ with the origin $0 \in \mathcal{V}_1$. The rigorous analysis is done in~\S\ref{sec:existence}. We compute the linearization of \eqref{eq:Lepsilon0} at $0$ when $\epsilon = 0$
$$
\delta_0\textbf{L}^0 = \delta_0\textbf{L}_1\oplus\delta_0\textbf{L}_2^0\oplus\delta_0\textbf{L}_3\oplus\delta_0\textbf{L}_4\colon \mathcal{V}_1 \to \mathcal{V}_2
$$
and study its mapping properties (see~Lemma~\ref{lemma:deltaL0} and~\eqref{eq:Laplacian0}). By standard properties of the HYM equation, the higher order derivatives in the linear operators $\delta_0\textbf{L}_1$ and $\delta_0\textbf{L}_4$ come from the Laplacian
\begin{equation}\label{eq:LaplacianAsmooth}
\Delta_A\colon \Omega_\mathbb{R}^0(\operatorname{ad}_0E) \to \Omega_\mathbb{R}^0(\operatorname{ad}_0E),
\end{equation}
induced by the Chern connection of $h_0$ on $E$ and the Laplacian
\begin{equation}\label{eq:Laplacian0smooth}
\Delta_0\colon \Omega_\mathbb{R}^0(\operatorname{ad}TX) \to \Omega_\mathbb{R}^0(\operatorname{ad}TX),
\end{equation}
induced by the Chern connection of $\omega_0$ on $TX$, respectively. Here we identify the codomain in \eqref{eq:LaplacianAsmooth} (resp. in \eqref{eq:Laplacian0smooth}) with $\Omega_\mathbb{R}^6(\operatorname{ad}_0E)$ (resp. $\Omega_\mathbb{R}^6(\operatorname{ad}TX)$) in \eqref{eq:Lepsilonsplit} via multiplication by the volume form $\frac{\omega^3_0}{3!}$. A crucial ingredient in our argument is the previous study in~\cite{LiYau} of the linear operator
\begin{equation}\label{eq:operatorT0}
\textbf{T} := \delta_0\textbf{L}_2^0\oplus\delta_0\textbf{L}_3.
\end{equation}
To apply the implicit function theorem we need that the differential of our operator induces an isomorphism at $0\in \mathcal{V}_1$, so we redefine \eqref{eq:Lepsilon0} as follows: first, we define the domain $\mathcal{W}_1 \subset \mathcal{V}_1$ and codomain $\mathcal{W}_2 \subset \mathcal{V}_2$ of the new operator. Let
$$
\mathcal{W}_1 := \mathcal{U} \cap (0 \oplus \operatorname{Ker} \textbf{T} \oplus \operatorname{Ker} \Delta_0)^\perp,
$$
where $(0 \oplus \operatorname{Ker} \textbf{T} \oplus \operatorname{Ker} \Delta_0)^\perp$ denotes the orthogonal complement in $\mathcal{V}_1$ of the direct sum of Kernels $0 \oplus \operatorname{Ker} \textbf{T} \oplus \operatorname{Ker} \Delta_0$ corresponding to \eqref{eq:LaplacianAsmooth}, \eqref{eq:Laplacian0smooth} and \eqref{eq:operatorT0}. Note here that $\operatorname{Ker} \Delta_0$ is given by the covariant constant sections
in $\Omega_\mathbb{R}^0(\operatorname{ad}TX)$, which correspond to the Lie algebra of the centralizer of the holonomy group of $\nabla_0$ in $U(3)$. Note also that, since $E$ is stable, $\operatorname{Ker} \Delta_A = 0$. By previous analysis in~\cite[Proposition~3.3]{LiYau}, the range $\operatorname{Im} \textbf{T}$ of \eqref{eq:operatorT0} is closed and so, by the ellipticity of $\Delta_A$ and $\Delta_0$,
$$
\mathcal{W}_2 := \operatorname{Im} \Delta_A \oplus \operatorname{Im} \textbf{T} \oplus \operatorname{Im} \Delta_0,
$$
defines a closed subspace of $\mathcal{V}_2$. Then, we consider the orthogonal projection $\Pi_{\mathcal{W}_2}\colon \mathcal{V}_2 \to \mathcal{W}_2$ and define a new operator
$$
\textbf{M}^\epsilon\colon \mathcal{W}_1 \to \mathcal{W}_2,
$$
by composition $\textbf{M}^\epsilon:=\Pi_{\mathcal{W}_2} \circ \textbf{L}^\epsilon_{|\mathcal{W}_1}$. In Proposition~\ref{prop:deltaL0bis} we prove that the differential of $\textbf{M}^0$ induces an isomorphism at $0 \in \mathcal{W}_1$ which allows to apply the implicit function theorem in the proof of Theorem~\ref{thm:existence0}. This provides solutions of $\textbf{M}^\epsilon = 0$ with $\epsilon >0$ which are smooth by standard elliptic regularity (see Lemma~\ref{lemma:regularity}). An explicit description of $\textbf{M}^\epsilon$ (see~\eqref{eq:Mepsilonbis}) shows that a solution $(h,\omega,\dbar_{_{TX}})$ of $\textbf{M}^\epsilon = 0$ corresponds to a solution of~\eqref{stromsys}-\eqref{anom} (and so of the Strominger system) provided that $c_2(E) = c_2(X)$. Moreover, such a triple satisfies
$$
\textbf{L}_4(h,\omega,\dbar_{_{TX}}) \in \operatorname{Ker} \Delta_0 \cdot \omega^3_0.
$$
When $(X,\omega_0)$ has strict $SU(3)$ holonomy, $\operatorname{Ker} \Delta_0$ reduces to the constant endomorphisms $i\mathbb{R}\operatorname{Id}$ which allows to prove that actually $\textbf{L}_4(h,\omega,\dbar_{_{TX}}) = 0$ in the proof of Theorem~\ref{thm:existence0}.

\section{Existence of solutions}
\label{sec:existence}

Let $X$ be a compact complex manifold endowed with a nowhere vanishing $(3,0)$-holomorphic form $\Omega$. Let $E$ be a holomorphic vector bundle over $X$ with rank $r$ and $c_1(E) = 0$. We fix a Hermitian form $\omega_0$ on $X$ and a reference Hermitian metric $h_0$ on $E$. Using the same notation as in~\S\ref{sec:methodresult}, given $\epsilon \in \mathbb{R}$ we define a differential operator
\begin{equation}\label{eq:Lepsilonsmooth}
\textbf{L}^\epsilon \colon \mathcal{H}(E)_1 \times \mathcal{H}(X) \times \mathcal{D}(TX) \to \Omega^6_\mathbb{R}(\operatorname{ad}_0E)\oplus\Omega_{\mathbb{R}}^{2,2}(X)\oplus\Omega_{\mathbb{R}}^1(X)\oplus \Omega^6_\mathbb{R}(\operatorname{ad}TX),
\end{equation}
as the direct sum $\textbf{L}^\epsilon = \textbf{L}_1\oplus\textbf{L}_2^\epsilon\oplus\textbf{L}_3\oplus\textbf{L}_4$ (see~\eqref{eq:Lepsilonsplit}). We identify now the zero locus of $L^\epsilon$, with $\epsilon = 0$, relying on previous results on~\cite{Strom} and~\cite{IvanovPapadopoulos}. Let $$(h,\omega,\dbar_{_{TX}}) \in \mathcal{H}(E)_1 \times \mathcal{H}(X) \times \mathcal{D}(TX).$$
\begin{lemma}\label{lemma:L0}
$\operatorname{\mathbf{L}}^0(h,\omega,\dbar_{_{TX}}) = 0$ if and only if $\omega$ is the K\"ahler form of a Calabi-Yau metric on $X$, $h$ is Hermite-Einstein with respect to $\omega$ and $\dbar_{_{TX}} = \dbar_0$ up to gauge transformation on $(TX,h_\omega)$.
\end{lemma}
\begin{proof}
Note first that $\textbf{L}_1(h,\omega) = 0$ if and only if $h$ is Hermite-Einstein with respect to $\omega$. We start with the `if part' of the statement. Since $\omega$ is Ricci flat, $\Omega$ is parallel and so has constant norm. Hence, $d\omega = 0$ implies $\textbf{L}_2^0(h,\omega,\dbar_{_{TX}}) = \textbf{L}_3(\omega) =  0$. The Chern connection of $\omega$ with respect to $\dbar_0$ is HYM and so $\textbf{L}_4(\omega,\dbar_{_{TX}}) = 0$ for any $\dbar_{_{TX}}$ gauge equivalent to $\dbar_0$. For the `only if part', $\textbf{L}_3(\omega) = 0$ implies that the Bismut connection of $\omega$ has $SU(3)$-holonomy (see~\cite[Section~II]{Strom}). Moreover, $\textbf{L}_2^0(\omega) = 0$ implies that $\partial\dbar\omega = 0$. Applying now~\cite[Corollary~4.7]{IvanovPapadopoulos}, the Bismut connection and the Chern connection coincide and $\omega$ is the K\"ahler form of a Calabi-Yau metric on $X$. Since $\textbf{L}_4(\omega,\dbar_{_{TX}}) = 0$ and the Chern connection of the Calabi-Yau metric is HYM, uniqueness implies that $\dbar_{_{TX}} = \dbar_0$ up to gauge transformation on $(TX,h_\omega)$.
\end{proof}

In the sequel we assume that $(X,\omega_0)$ is a Calabi-Yau threefold and that $E$ is stable with respect to $[\omega_0]$, with $h_0 \in \mathcal{H}(E)_1$ the unique Hermite-Einstein metric given by the Donaldson-Uhlenbeck-Yau Theorem. By Lemma~\ref{lemma:L0}, $(h_0,\omega_0,\dbar_0)$ is a solution of $\textbf{L}^0 = 0$. We compute now
$$
\delta_0\textbf{L}^0 = \delta_0\textbf{L}_1\oplus\delta_0\textbf{L}_2^0\oplus\delta_0\textbf{L}_3\oplus\delta_0\textbf{L}_4,
$$
the differential of $\textbf{L}^0$ at $(h_0,\omega_0,\dbar_0)$, and study its mapping properties. Let $D_A$ and $\nabla_0$ be the covariant derivative of $h_0$ on $E$ and $h_{\omega_0}$ on $(TX,\dbar_0)$ and consider their standard decompositions
$$
D_A = \partial_A + \dbar_A \qquad \textrm{and} \qquad \nabla_0 = \partial_0 + \dbar_0,
$$
with respect to the complex structure on $X$. We use the canonical isomorphisms
$$
T_{h_0}\mathcal{H}(E)_{1} \cong i\cdot\Omega_\mathbb{R}^0(\operatorname{ad}_0E), \ \ T_{\omega_0}\mathcal{H}(X) \cong \Omega_{\mathbb{R}}^{1,1}(X),
$$
the inclusion $T_{\dbar_0}\mathcal{D}(TX)\subset \Omega^{0,1}({\rm End}(TX))$ and denote $0 < c := \|\Omega\|_{\omega_0} \in \mathbb{R}$. The elements in $T_{h_0}\mathcal{H}(E)_{1}$, $T_{\omega_0}\mathcal{H}(X)$ and $T_{\dbar_0}\mathcal{D}(TX)$ are denoted respectively by $\delta h$, $\delta \omega$ and $\delta \dbar$. Given $\delta\omega \in \Omega_{\mathbb{R}}^{1,1}(X)$, we denote by $\delta h_\omega$ the symmetric endomorphism of $(TX,h_{\omega_0})$ determined by the identification between Hermitian metrics and determinant one symmetric endomorphism given by $\langle\cdot,\cdot\rangle_{h_\omega}=\langle h_{\omega}\cdot,\cdot \rangle_{h_{\omega_0}}$ (cf.~\eqref{eq:omegahomega}). Let
\begin{equation}\label{eq:formulaLaplacian}
\Delta_A := D_A^* D_A + D_A D_A^*\colon \Omega_\mathbb{R}^0(\operatorname{ad}_0E) \to \Omega_\mathbb{R}^0(\operatorname{ad}_0E),
\end{equation}
denote the Laplacian operator induced by $A$, where $D_A^*$ is the adjoint of $D_A$ with respect to $h_0$ and $\omega_0$. In the sequel we use the standard notation $d^c = i(\dbar - \partial)$.
\begin{lemma}\label{lemma:deltaL0}
The map $\delta_0\textbf{L}^0$ is given by
\begin{align*}
\delta_0\textbf{L}_1(\delta h,\delta\omega) & = \Delta_A(-i\delta h)\frac{\omega_0^3}{3!} + 2 F_{h_0}\wedge \delta \omega \wedge \omega_0,\\
\delta_0\textbf{L}_2^0(\delta\omega) & = \frac{1}{2}dd^c(\delta\omega),\\
\delta_0\textbf{L}_3(\delta\omega) & =  c \; d^*(2\delta\omega - (\omega_0,\delta\omega)\omega_0),\\
\delta_0\textbf{L}_4(\delta\omega,\delta \dbar) & = \nabla_0(\delta\dbar - (\delta\dbar)^*)\wedge \omega_0^2 + \dbar_0\partial_0(\delta h_\omega)\wedge\omega_0^2 + 2 R_{\nabla_0}\wedge \delta \omega \wedge \omega_0,
\end{align*}
where $(\delta\dbar)^*$ and $(\omega_0,\delta\omega)$ denote the adjoint of $\delta\dbar$ and the inner product on forms with respect to $\omega_0$, respectively.
\end{lemma}
\begin{proof}
The formula for $\delta_0\textbf{L}_2^0$ is obvious, since $dd^c$ is a linear operator. The $\delta_0\textbf{L}_3$ term has been computed previously in \cite[Proposition~3.2]{LiYau}, where $c = \|\Omega\|_{\omega_0}$. For the computation of $\delta_0\textbf{L}_4$, recall that an infinitesimal change $\delta\nabla$ in the connection $\nabla_0$ produces an infinitesimal change of the curvature $\delta R(\delta\nabla) = \nabla_0(\delta\nabla)$. Then, since
$$
\delta_0\textbf{L}_4(\delta\omega,\delta \dbar) = \delta R(\delta\dbar)\wedge \omega_0^2 + \delta R(\delta h_\omega)\wedge\omega_0^2 + 2 R_{\nabla_0}\wedge \delta \omega \wedge \omega_0,
$$
the result follows from $\delta\nabla(\delta\dbar) = \delta\dbar - (\delta\dbar)^*$ and $\delta\nabla(\delta h_\omega) = \partial_0\delta h_\omega$. A similar computation shows that
$$
\delta_0\textbf{L}_1(\delta h,\delta\omega) = \dbar_A\partial_A(\delta h)\wedge\omega_0^2 + 2 F_{h_0}\wedge \delta \omega \wedge \omega_0,
$$
and using the K\"ahler identities on $(X,\omega_0)$ we obtain (see the proof of~\cite[Proposition~3]{Don})
\begin{equation}\label{eq:Laplacianconnection}
\dbar_A\partial_A(\delta h)\wedge \omega_0^2 = \Delta_A(-i\delta h)\frac{\omega_0^3}{3!},
\end{equation}
proving the claim.
\end{proof}

In order for our deformation map to have domain and codomain which are (locally) isomorphic, we consider the map
$$
\Omega_{\mathbb{R}}^0(\operatorname{ad}TX)\to \mathcal{D}(TX)\colon a \mapsto \dbar_a:=e^{{ia}}\dbar_0 e^{-ia},
$$
given by the standard action of $e^{ia}$ on $\dbar_0$, which is well defined since $\dbar_a^2 = 0$. Note that any element in the image of this map is isomorphic to $\dbar_0$ via a complex gauge transformation. Furthermore, we use the isomorphism (cf. \eqref{eq:omegahomega})
$$
\Omega_{\mathbb{R}}^0(\operatorname{ad}_0E) \to \mathcal{H}(E)_1\colon b \to \langle h_b\cdot,\cdot\rangle_{h_0},
$$
where $h_b = e^{ib}$ to define a new operator by
\begin{equation}\label{eq:Lepsilonsmoothbis}
(b,\omega,a) \mapsto \textbf{L}^\epsilon(h_b,\omega,\dbar_a),
\end{equation}
with domain $\Omega_{\mathbb{R}}^0(\operatorname{ad}_0E) \times \mathcal{H}(X) \times \Omega_{\mathbb{R}}^0(\operatorname{ad}TX)$ and codomain as in \eqref{eq:Lepsilonsmooth}. For simplicity, the operator defined by \eqref{eq:Lepsilonsmoothbis} will also be denoted by $\textbf{L}^\epsilon$. Note that the canonical solution of $\textbf{L}^0 = 0$ is now given by $(0,\omega_0,0)$ and that the expressions for $\delta_0\textbf{L}^0_2$ and $\delta_0\textbf{L}_3$ in the differential $\delta_0\textbf{L}^0$ at $(0,\omega_0,0)$ remain unchanged (see Lemma~\ref{lemma:deltaL0}), while the new expression for $\delta_0\textbf{L}_1$ is obtained by the change of variable $\delta h = ib$. To compute the new expression for $\delta_0\textbf{L}_4$, let $(\mu,a) \in \Omega_{\mathbb{R}}^{1,1}(X)\oplus \Omega_{\mathbb{R}}^0(\operatorname{ad}TX)$. Then,
\begin{equation}\label{eq:Laplacian0}
\begin{split}
\delta_0\textbf{L}_4(\mu,a) & = \nabla_0(-i\dbar_0a - (-i\dbar_0a)^*)\wedge \omega_0^2 + \dbar_0\partial_0h_\mu\wedge\omega_0^2 + 2 R_{\nabla_0}\wedge\mu\wedge \omega_0\\
& = 2i\dbar_0\partial_0a\wedge \omega_0^2 + \dbar_0\partial_0h_\mu\wedge\omega_0^2 + 2 R_{\nabla_0}\wedge\mu\wedge \omega_0\\
& = \frac{1}{3}\Delta_{0}(a - \frac{ih_\mu}{2})\omega_0^3 + 2R_{\nabla_0}\wedge\mu\wedge \omega_0,
\end{split}
\end{equation}
where $h_\mu$ denotes the symmetric endomorphisms of $(TX,h_{\omega_0})$ determined by $\mu$ and $\Delta_{0}$ denotes the Laplacian operator induced by the connection $\nabla_0$ (cf.~\eqref{eq:formulaLaplacian}). Here we have used the equality
$$
(-i\partial_0\dbar_0a + i\dbar_0\partial_0 a)\wedge \omega_0^2 = 2i\partial_0\dbar_0a\wedge \omega_0^2 - i R_{\nabla_0}a\wedge\omega_0^2 = 2i\partial_0\dbar_0a\wedge \omega_0^2,
$$
together with $(i\dbar_0a)^* = i\partial_0a$ and the analogue of formula~\eqref{eq:Laplacianconnection} for $\Delta_0$.

To study the mapping properties of $\delta_0 \textbf{L}^0$ we introduce Sobolev spaces. Given a smooth real vector bundle $W$ over $X$ endowed with a metric, we consider the real $L^2$ pairings on $\Omega_{\mathbb{R}}^m(W)$ and $\Omega_{\mathbb{R}}^{p,q}(W)$ provided by the metric on $W$ and the K\"ahler form $\omega_0$. Let $\Omega_{\mathbb{R}}^m(W)_{L^2_{k}}$ and $\Omega_{\mathbb{R}}^{p,q}(W)_{L^2_{k}}$ be the respective Sobolev completions, consisting of sections whose first $k$-distributional derivatives are in $L^2$. Since $\dim_{\mathbb{C}} X = 3$, the Sobolev embedding theorem (see e.g.~\cite{Aubin}) states that if $k > 3 + l$ then the elements in $\Omega_{\mathbb{R}}^m(W)_{L^2_{k}}$ and $\Omega_{\mathbb{R}}^{p,q}(W)_{L^2_{k}}$ are of class $C^l$. Furthermore, for $k > 3$ these spaces are Banach algebras. Let $k > 3$ and consider the Hilbert spaces
\begin{equation}\label{eq:V1}
\mathcal{V}_1 := \Omega_\mathbb{R}^0(\operatorname{ad}_0E)_{L^2_k} \oplus \Omega_{\mathbb{R}}^{1,1}(X)_{L^2_k}\oplus \Omega_\mathbb{R}^0(\operatorname{ad}TX)_{L^2_k},
\end{equation}
and
\begin{equation}\label{eq:V2}
\mathcal{V}_2 := \Omega^6_\mathbb{R}(\operatorname{ad}_0E)_{L^2_{k-2}}\oplus\Omega_{\mathbb{R}}^{2,2}(X)_{L^2_{k-2}}\oplus\Omega_{\mathbb{R}}^1(X)_{L^2_{k-1}}\oplus \Omega^6_\mathbb{R}(\operatorname{ad}TX)_{L^2_{k-2}},
\end{equation}
where the metric on the bundles $\operatorname{ad}_0E$ and $\operatorname{ad}TX$ is given by $-\operatorname{tr}$. Let $\mathcal{U} \subset \mathcal{V}_1$ be open subset given by triples $(h,\mu,a)\in \mathcal{V}_1$ such that $\mu$ is positive definite. By the Banach algebra properties of the previous spaces for $k > 3$, the operator \eqref{eq:Lepsilonsmooth} induces a well defined map
\begin{equation}\label{eq:Lepsilonsobolev}
\textbf{L}^\epsilon\colon\mathcal{U}\to\mathcal{V}_2.
\end{equation}
A straightforward computation of the differential of $\textbf{L}^\epsilon$ at an arbitrary point of $\mathcal{U}$ shows that \eqref{eq:Lepsilonsobolev} is a $C^1$ map (cf. Lemma~\ref{lemma:deltaL0} and \cite[Proposition~3.2]{LiYau}).

To apply the implicit function theorem we need that the differential of our map induces an isomorphism at $(0,\omega_0,0)$, so we redefine \eqref{eq:Lepsilonsobolev} as follows. Let
\begin{equation}\label{eq:Laplacianssobolev}
\Delta_A\colon \Omega^0_\mathbb{R}(\operatorname{ad}_0E)_{L^2_k} \to \Omega^0_\mathbb{R}(\operatorname{ad}_0E)_{L^2_{k-2}}, \qquad \Delta_0\colon \Omega^0_\mathbb{R}(\operatorname{ad}TX)_{L^2_k} \to \Omega^0_\mathbb{R}(\operatorname{ad}TX)_{L^2_{k-2}},
\end{equation}
be the operators induced for any $k$ by the Laplacians $\Delta_A$ and $\Delta_0$, respectively. Note that, since $E$ is stable, $\Delta_A$ in \eqref{eq:Laplacianssobolev} is an isomorphism. Consider the differential operator
$$
\textbf{T}\colon \Omega_{\mathbb{R}}^{1,1}(X)_{L^2_{k}} \to \Omega_{\mathbb{R}}^{2,2}(X)_{L^2_{k-2}}\oplus\Omega_{\mathbb{R}}^1(X)_{L^2_{k-1}},
$$
given by
\begin{equation}\label{eq:operatorT}
\textbf{T}(\mu):= \delta_0 \textbf{L}_2\oplus\delta_0 \textbf{L}_3 \;(\mu) = (\frac{1}{2}dd^c\mu,c \;d^*(2\mu - (\mu,\omega_0)\omega_0)).
\end{equation}
Note here that $\textbf{T}$ only depends on $(X,\omega_0)$. We now recall a previous result due to Li-Yau concerning \eqref{eq:operatorT} (contained in the proof of \cite[Proposition~3.3]{LiYau}) which is crucial for our argument. Let $\operatorname{Im}(dd^c) \subset \Omega_{\mathbb{R}}^{2,2}(X)_{L^2_{k-2}}$ and $\operatorname{Im}(d^*)\subset \Omega_{\mathbb{R}}^1(X)_{L^2_{k-1}}$ be the closed subspaces given by the range of the differential operators
$$
dd^c \colon \Omega_{\mathbb{R}}^{1,1}(X)_{L^2_{k}} \to \Omega_{\mathbb{R}}^{2,2}(X)_{L^2_{k-2}} \qquad \textrm{and}\qquad d^* \colon \Omega_{\mathbb{R}}^{1,1}(X)_{L^2_{k}} \to \Omega_{\mathbb{R}}^1(X)_{L^2_{k-1}}.
$$
\begin{prop}[Li \& Yau]\label{prop:LiYau}
The operator $\operatorname{\mathbf{T}}$ has closed range $\operatorname{Im}\operatorname{\mathbf{T}}$ given by $ \operatorname{Im}(dd^c) \oplus \operatorname{Im}(d^*)$ and its kernel $\operatorname{Ker}\operatorname{\mathbf{T}}\subset \Omega_{\mathbb{R}}^{1,1}(X)$ is the subspace of $\omega_0$-harmonic $(1,1)$ forms on $X$.
\end{prop}

Consider the closed subspace $0 \oplus \operatorname{Ker} \textbf{T} \oplus \operatorname{Ker} \Delta_0$ of $\mathcal{V}_1$ given by the direct sum of the Kernels of the operators in \eqref{eq:Laplacianssobolev} and \eqref{eq:operatorT}. We denote its orthogonal complement by $(0 \oplus \operatorname{Ker} \textbf{T}\oplus \operatorname{Ker}\Delta_0)^\perp$ and consider
\begin{equation}\label{eq:translation}
\textbf{t}_{\omega_0}\colon \mathcal{V}_1 \to \mathcal{V}_1,
\end{equation}
the translation operator defined by $v \mapsto v + (0,\omega_0,0)$. Consider also the closed affine subspace $\textbf{t}_{\omega_0}((0 \oplus \operatorname{Ker} \textbf{T}\oplus \operatorname{Ker}\Delta_0)^\perp) \subset \mathcal{V}_1$ and define
$$
\mathcal{W}_1 :=  \mathcal{U} \cap \textbf{t}_{\omega_0}((0 \oplus \operatorname{Ker} \textbf{T}\oplus \operatorname{Ker}\Delta_0)^\perp).
$$
Trivially, $(0,\omega_0,0) \in \mathcal{W}_1$. By ellipticity of $\Delta_A$ and $\Delta_0$ in \eqref{eq:Laplacianssobolev} we have closed subspaces
$$
\operatorname{Im} \Delta_A \subset \Omega^6_\mathbb{R}(\operatorname{ad}_0E)_{L^2_{k-2}} \qquad \operatorname{Im}\Delta_0\subset \Omega^6_\mathbb{R}(\operatorname{ad}TX)_{L^2_{k-2}},
$$
where we identify the codomains in \eqref{eq:Laplacianssobolev} with $\Omega^6_\mathbb{R}(\operatorname{ad}_0E)_{L^2_{k-2}}$  and $\Omega^6_\mathbb{R}(\operatorname{ad}TX)_{L^2_{k-2}}$ via multiplication by the volume form $\frac{\omega^3_0}{3!}$. Hence, by Proposition~\ref{prop:LiYau} we have a closed subspace of $\mathcal{V}_2$ given by
$$
\mathcal{W}_2 := \operatorname{Im} \Delta_A \oplus \operatorname{Im} \textbf{T} \oplus \operatorname{Im} \Delta_0 \subset \mathcal{V}_2.
$$
Let $\Pi_{\mathcal{W}_2}\colon \mathcal{V}_2 \to \mathcal{W}_2$ be the associated orthogonal projection. The map we are going to use in the perturbative argument is
\begin{equation}\label{eq:Mepsilon}
\textbf{M}^\epsilon\colon \mathcal{W}_1 \to \mathcal{W}_2, \quad \textrm{defined by} \quad \textbf{M}^\epsilon:=\Pi_{\mathcal{W}_2} \circ \textbf{L}^\epsilon_{|\mathcal{W}_1},
\end{equation}
where $\circ$ denotes composition and $\textbf{L}^\epsilon_{|\mathcal{W}_1}$ is the restriction of \eqref{eq:Lepsilonsobolev} to $\mathcal{W}_1$. We find now a more convenient expression for \eqref{eq:Mepsilon}. Consider the orthogonal projections
\begin{equation}\label{eq:projectorsTdelta0}
\Pi_{\textbf{T}}\colon \Omega_{\mathbb{R}}^{2,2}(X)_{L^2_{k-2}}\oplus\Omega_{\mathbb{R}}^1(X)_{L^2_{k-1}} \to \operatorname{Im}\textbf{T},
\end{equation}
\begin{equation}\label{eq:PiKer}
\Pi_{\Delta_0}\colon \Omega^6_\mathbb{R}(\operatorname{ad}TX)_{L^2_{k-2}} \to \operatorname{Im}\Delta_0.
\end{equation}
Since $\Delta_A$ induces an isomorphism in \eqref{eq:Laplacianssobolev} we have
\begin{equation}\label{eq:Mepsilonbis}
\textbf{M}^\epsilon = \textbf{L}_1\oplus \Pi_{\textbf{T}} \circ (\textbf{L}_2^\epsilon\oplus\textbf{L}_3) \oplus \Pi_{\Delta_0} \circ \textbf{L}_4.
\end{equation}
Further, if the condition $c_2(E) = c_2(X)$ is satisfied then $\operatorname{Im}\textbf{L}_2^\epsilon \subset \operatorname{Im}(dd^c)$ and so we have $\Pi_{\textbf{T}}\circ ( \textbf{L}_2^\epsilon\oplus\textbf{L}_3) = \textbf{L}_2^\epsilon\oplus\textbf{L}_3$. The choice of $\textbf{M}^\epsilon$ is now justified by the next result, which is independent of the mentioned topological condition.
\begin{prop}\label{prop:deltaL0bis}
The differential $\delta_0\textbf{M}^0$ of \eqref{eq:Mepsilon} at $0 \in \mathcal{W}_1$ is an isomorphism.
\end{prop}
\begin{proof}
First, we prove that $\delta_0\textbf{M}^0\colon (0 \oplus \operatorname{Ker} \textbf{T}\oplus \operatorname{Ker}\Delta_0)^\perp \to \mathcal{W}_2$ is surjective. For this, note that by \eqref{eq:Mepsilonbis} we have
\begin{equation}\label{eq:deltaM00}
\delta_0\textbf{M}^0 = \delta_0\textbf{L}_1 \oplus \textbf{T} \oplus \Pi_{\Delta_0} \circ \delta_0\textbf{L}_4.
\end{equation}
Moreover, it follows from Lemma~\ref{lemma:deltaL0} and formula \eqref{eq:Laplacian0} that
\begin{equation}\label{eq:deltaL1L4}
\begin{split}
\delta_0\textbf{L}_1(b,\mu) & = \Delta_A(b) \frac{\omega_0^3}{3!} + 2 F_{h_0}\wedge \mu \wedge \omega_0,\\
\Pi_{\Delta_0} \circ\delta_0\textbf{L}_4(\mu,a) & = \frac{1}{3}\Delta_{0}(a - \frac{ih_\mu}{2})\omega_0^3 + \Pi_{\Delta_0}(2R_{\nabla_0}\wedge\mu\wedge \omega_0),
\end{split}
\end{equation}
for any $(b,\mu) \in \Omega_\mathbb{R}^0(\operatorname{ad}_0E)_{L^2_k}\oplus\Omega_{\mathbb{R}}^{1,1}(X)_{L^2_{k}}$ and $(\mu,a) \in \Omega_{\mathbb{R}}^{1,1}(X)_{L^2_{k}}\oplus \Omega_\mathbb{R}^0(\operatorname{ad}TX)_{L^2_k}$. Hence, by~\eqref{eq:deltaL1L4} and stability of $E$ we have that
\begin{equation}\label{eq:deltaL1T}
\delta_0 \textbf{L}_1 \oplus \textbf{T}\colon (0 \oplus \operatorname{Ker} \textbf{T})^\perp\to \operatorname{Im} \Delta_A \oplus \operatorname{Im} \textbf{T},
\end{equation}
is surjective, where $(0 \oplus \operatorname{Ker} \textbf{T})^\perp \subset \Omega_\mathbb{R}^0(\operatorname{ad}_0E)_{L^2_{k}}\oplus \Omega_{\mathbb{R}}^{1,1}(X)_{L^2_{k}}$. Similarly, it follows from~\eqref{eq:deltaL1L4} that
\begin{equation}\label{eq:deltaTL4}
\textbf{T} \oplus \Pi_{\Delta_0} \circ \delta_0\textbf{L}_4\colon (\operatorname{Ker} \textbf{T}\oplus \operatorname{Ker} \Delta_0)^\perp\to \operatorname{Im} \textbf{T} \oplus \operatorname{Im} \Delta_0
\end{equation}
is surjective, where $(\operatorname{Ker} \textbf{T}\oplus \operatorname{Ker} \Delta_0)^\perp \subset \Omega_{\mathbb{R}}^{1,1}(X)_{L^2_{k}}\oplus \Omega_\mathbb{R}^0(\operatorname{ad}TX)_{L^2_k}$. The surjectivity of $\delta_0\textbf{M}^0$ is a direct consequence of the surjectivity of \eqref{eq:deltaL1T} and \eqref{eq:deltaTL4}. To prove the injectivity, let
$$
(b,\mu,a)\in(0 \oplus \operatorname{Ker}\textbf{T} \oplus \operatorname{Ker}\Delta_0)^\perp
$$
such that $\delta_0\textbf{M}^0(b,\mu,a) = 0$. Then, by \eqref{eq:deltaM00} we have $\textbf{T}(\mu) = 0$, and so $\mu = 0$, and it follows from \eqref{eq:deltaL1L4} that
$\Delta_A(b) = 0$ and $\Delta_0(a) = 0$.
Finally, $b= 0$ by stability of $E$ and $a = 0$ since $a \in \operatorname{Ker} \Delta_0^\perp$, which proves the claim.
\end{proof}

Before we go into the proof of the main theorem we prove a technical lemma concerning the regularity of the solutions of $\textbf{M}^\epsilon = 0$. We assume $c_2(E) = c_2(X)$, so that $\Pi_{\textbf{T}}\circ ( \textbf{L}_2^\epsilon\oplus\textbf{L}_3) = \textbf{L}_2^\epsilon\oplus\textbf{L}_3$.

\begin{lemma}\label{lemma:regularity}
If $k > 5$ and $\epsilon > 0$ is small enough, then if $h$, $\omega$ and $a$ satisfy $\textbf{M}^\epsilon(h,\omega,a) = 0$ they are of class $C^{\infty}$.
\end{lemma}
\begin{proof}
As $k > 5$, we have that $h$, $\omega$ and $a$ are of class $C^{2,\alpha}$ for any $0 < \alpha < 1$. We argue by induction on $m$. For this, assume $h$, $\omega$ and $a$ are of class $C^{m,\alpha}$. Then, by \eqref{eq:Mepsilonbis}, $h$ and $a$ satisfy
\begin{equation}\label{eq:lemmaregular1}
\textbf{L}_1(h,\omega) = 0 \qquad \Pi_{\Delta_0}\circ\textbf{L}_4(\omega,a) = 0.
\end{equation}
In holomorphic coordinates for the bundle $E$, the first equation implies the equality
\begin{equation}\label{eq:lemmaregular2}
\dbar\partial h\wedge\omega^2 = - h \; \dbar(h^{-1}) \wedge \partial h\wedge\omega^2,
\end{equation}
and since the LHS defines a linear elliptic operator of order two with $C^{m,\alpha}$ coefficients $h$ is of class $C^{m+1,\alpha}$ (see \cite[Theorem~3.55]{Aubin}). We claim that $e^{2ia}h_\omega$ is also of class $C^{m+1,\alpha}$. To see this, note that $R_{h_\omega,\dbar_a} = e^{ia}R_{e^{2ia}h_\omega,\dbar_0}e^{-ia}$ and that the second equation in \eqref{eq:lemmaregular1} is equivalent to
\begin{equation}\label{eq:lemmaregular3}
\textbf{L}_4(\omega,a) = (\operatorname{Id} - \Pi_{\Delta_0})\textbf{L}_4(\omega,a).
\end{equation}
Therefore, the claim follows, similarly as for $h$, by taking local holomorphic coordinates for $\dbar_0$ on $TX$. Note here that $(\operatorname{Id} - \Pi_{\Delta_0})$ equals the orthogonal projection on the space
\begin{equation}\label{eq:lemmaregular2}
\operatorname{Ker}\Delta_0 \frac{\omega_0^3}{3!} \subset \Omega^6_\mathbb{R}(\operatorname{ad}TX)_{L^2_{k-2}},
\end{equation}
whose elements are of class $C^\infty$ by elliptic regularity of $\Delta_0$. Hence, the proof reduces to prove that $\omega$ is of class $C^{m+1,\alpha}$. For this, note first that from $\textbf{L}_2^\epsilon(h,\omega,a) = 0$ one can easily obtain now that $\partial \dbar\omega$ is of class $C^{m-1,\alpha}$, as $h$ and $e^{2ia}h_\omega$ are of class $C^{m+1,\alpha}$. Write now $\omega = \omega_0 + \mu$, for a $(1,1)$-form $\mu$ on $X$. By the Hodge decomposition for the $\partial\dbar$-operator with respect to $\omega_0$ we have
$$
\mu = \sigma + \partial\dbar \phi + \mu_H,
$$
where $\mu_H$ is $\partial\dbar$-harmonic, $\sigma = (\partial\dbar)^*\partial\dbar\psi$ and $\phi = (\partial\dbar)^*\psi$ for some $(1,1)$-form $\psi$ of class $C^{m+4,\alpha}$. Hence, as $\partial\dbar\omega = \partial\dbar\sigma$ and $\Delta_{\partial\dbar,\omega_0}\sigma = (\partial\dbar)^*\partial\dbar \sigma$ we conclude that $\sigma$ is of class $C^{m+1,\alpha}$. Finally, from $\textbf{L}_3(\omega) = 0$ it follows easily that $\|\Omega\|_\omega$ is of class $C^{m+1,\alpha}$ since
$$
d(\|\Omega\|_\omega)\wedge \omega^2 = - \|\Omega\|_\omega 2d\omega'\wedge\omega,
$$
where $\omega' = \omega_0 + \sigma + \mu_H$ is of class $C^{m+1,\alpha}$. Since $\epsilon$ is small, $\omega'$ is non-degenerate and hence $\phi$ is of class $C^{m+3,\alpha}$ by regularity of the Monge-Amp\`ere equation (see \cite[Lemma 17.16]{GiTru}). This concludes the proof.
\end{proof}

We are now prepared to prove Theorem~\ref{thm:existence0}. Recall from \S\ref{sec:methodresult} that $(h_0,\omega_0,\dbar_0)$ is the canonical solution of the $\epsilon$-system with $\epsilon = 0$, identified with the large radius limit $\lambda \to \infty$ of the Strominger system and \eqref{Rinst}.

\begin{proof}
By Proposition~\ref{prop:deltaL0bis}, $\textbf{M}^\epsilon\colon\mathcal{W}_1\to\mathcal{W}_2$ defines a $C^1$ map whose Fr\'echet derivative at $(0,\omega_0,0)$ with $\epsilon = 0$ is an isomorphism. Applying the implicit function theorem (see e.g.~\cite[Theorem~3.10]{Aubin}), there exists $\epsilon_0 > 0$ and a $C^1$ curve
$$
]-\epsilon_0,\epsilon_0[ \to \mathcal{W}_1\colon \epsilon \mapsto (b_\epsilon,\omega_\epsilon,a_\epsilon)
$$
such that $\textbf{M}^\epsilon(b_\epsilon,\omega_\epsilon,a_\epsilon) = 0$ for all $\epsilon \in ]-\epsilon_0,\epsilon_0[$. Since $c_2(E) = c_2(X)$, by \eqref{eq:Mepsilonbis} we have
\begin{equation}\label{eq:thmidentities}
\textbf{L}_1\oplus \textbf{L}_2^\epsilon\oplus\textbf{L}_3\oplus \Pi_{\Delta_0}\circ\textbf{L}_4(h_\epsilon,\omega_\epsilon,a_\epsilon) = 0,
\end{equation}
and by Lemma~\ref{lemma:regularity} the three elements $b_\epsilon$, $\omega_\epsilon$ and $a_\epsilon$ are of class $C^\infty$. Setting now
$
\lambda := \frac{\alpha'}{\epsilon},
$
we obtain a $C^1$ curve $]\lambda_0,+\infty[ \to \lambda \to (h_\lambda,\omega_\lambda,\dbar_\lambda):= (h_{b_\epsilon},\lambda\omega_\epsilon,\dbar_{a_\epsilon})$ of solutions of the Strominger system which satisfies the first part of the statement. For the second, suppose that $(X,\omega_0)$ has holonomy equal to $SU(3)$. Then, $\operatorname{Ker} \Delta_0 = i \mathbb{R}\operatorname{Id}$ and we have the orthogonal projector
$$
\Pi:= (\operatorname{Id} - \Pi_{\Delta_0})\colon \Omega^6_\mathbb{R}(\operatorname{ad}TX)_{L^2_{k-2}} \to i\mathbb{R}\operatorname{Id}\frac{\omega_0^3}{3!}, a \to \frac{\langle a,i\operatorname{Id}\frac{\omega_0^3}{3!}\rangle_{L^2}}{\|i\operatorname{Id}\frac{\omega_0^3}{3!}\|^2_{L^2}}i\operatorname{Id}\frac{\omega_0^3}{3!}.
$$
Note that $\langle a,i\operatorname{Id}\frac{\omega_0^3}{3!}\rangle_{L^2} = \int_X i\operatorname{tr}a$ and $\|i\operatorname{Id}\frac{\omega_0^3}{3!}\|^2_{L^2} = r\operatorname{Vol}(\omega_0)$, where $r$ is the rank of $E$, so explicitly we have
\begin{equation}\label{eq:projection}
\Pi \circ \textbf{L}_4(\omega,a) = \frac{\int_X \operatorname{tr}i R_{\omega,\dbar_a}\wedge\omega^2}{r\operatorname{Vol}(\omega_0)}i\operatorname{Id}\frac{\omega_0^3}{3!},
\end{equation}
which may vary with $(\omega,a)$ since $\omega$ is not necessarily closed. Given $\epsilon \in ]-\epsilon_0,\epsilon_0[$, by~\eqref{eq:thmidentities} we have $d(\|\Omega\|_{\omega_\epsilon}\omega^2_\epsilon) = 0$ and $\Pi_{\Delta_0}\circ\textbf{L}_4(\omega_\epsilon,a_\epsilon) = 0$, which is equivalent to
\begin{equation}\label{eq:lemmaregular}
\textbf{L}_4(\omega_\epsilon,a_\epsilon) = \Pi \circ \textbf{L}_4(\omega_\epsilon,a_\epsilon).
\end{equation}
We multiply by $i\|\Omega\|_{\omega_\epsilon}$ on both sides of \eqref{eq:lemmaregular}, take the trace and integrate over $X$, therefore obtaining
\begin{align*}
0 = c_1(X) \cdot [\|\Omega\|_{\omega_\epsilon}\omega^2_\epsilon]& = \int_X \|\Omega\|_{\omega_\epsilon} i \operatorname{tr} \Pi \circ \textbf{L}_4(\omega_\epsilon,a_\epsilon)\\
& = - \frac{\int_X \operatorname{tr}i R_{\omega_\epsilon,\dbar_{a_\epsilon}}\wedge\omega^2_\epsilon}{\operatorname{Vol}(\omega_0)}\int_X \|\Omega\|_{\omega_\epsilon}\frac{\omega_0^3}{3!},
\end{align*}
where $\cdot$ denotes intersection product in cohomology. Here we have used $c_1(X) = 0$ and formula~\eqref{eq:projection}. Hence, since $\|\Omega\|_{\omega_\epsilon}$ is everywhere positive, we have
$$
\int_X \operatorname{tr} i R_{\omega_\epsilon,\dbar_{a_\epsilon}}\wedge\omega^2_\epsilon=0
$$
and so $\textbf{L}_4(\omega_\epsilon,a_\epsilon) = \Pi\circ\textbf{L}_4(\omega_\epsilon,a_\epsilon) = 0$ which proves the statement.
\end{proof}

\section{Examples}
In this section we give some known examples of stable vector bundles on Calabi-Yau threefolds satisfying
the hypothesis of Theorem~\ref{thm:existence0}.
\vskip .5cm
Let us note that there are presently no general results available concerning the geometry of moduli spaces of stable bundles (or sheaves) on Calabi-Yau threefolds which could provide solutions of the second Chern class constraint \eqref{eq:Anomallycohom0} and so give a large class of examples where Theorem~\ref{thm:existence0} applies. One way to find examples of stable bundles which satisfy the constraint is to explicitly construct them on a given Calabi-Yau threefold.

The probably most prominent threefold is the quintic Calabi-Yau threefold. A stable rank four vector bundle of vanishing first Chern class on the quintic which satisfies the Chern class constraint has been constructed in \cite{DiGr88} using the vector bundle construction of \cite{Maru78}. Another example of a stable rank four bundle on the quintic threefold which satisfies the constraint is given by a smooth deformation of tangent bundle and the trivial line bundle \cite{Huyb}. We note that the result of ~\cite{Huyb} combined with Theorem~\ref{thm:existence0} provides also a new proof of~\cite[Theorem~5.1]{LiYau} which moreover assures that (for the generic quintic) the solutions of the Strominger system satisfy also the equations of motion derived from the effective heterotic string action. An extensive search for vector bundles on the quintic Calabi-Yau threefold which satisfy the second Chern class constraint has been performed in \cite{DouZu} using the monad construction. These bundles have been further investigated in \cite{bram} proving stability of some rank three bundles.

Elliptically fibered Calabi-Yau threefolds with section have been also extensively investigated. In \cite{AnCu2} a class of stable extension bundles has been constructed which satisfy the second Chern class constraint. More generally, if the Calabi-Yau manifold admits a fibration structure, a natural procedure to describe moduli spaces of sheaves and bundles is to first construct them fiberwise and then to find an appropriate global description. This method has been successfully employed to construct stable vector bundle on elliptic fibrations (cf. Section 6) and on $K3$ fibrations. However, these bundles do not satisfy \eqref{eq:Anomallycohom0}, but only a generalized constraint (see \eqref{topfive}) which lead us to the topic of the next section.

\section{Further Directions}
\label{sec:remarks}

In more general string backgrounds (i.e., if a number of five-branes contributes to the compactification, cf. \cite{FMW97, WiStro}), the role of the integrability condition~\eqref{eq:Anomallycohom0} is played by the generalized cohomological condition
\begin{equation}\label{topfive}
c_2(X)=c_2(E)+[W],
\end{equation}
with $[W]$ an effective curve class. Three approaches to construct holomorphic $G$-bundles on elliptically fibered Calabi-Yau threefolds $X$ that are capable to satisfy \eqref{topfive} have been introduced in \cite{FMW97, FMW98, FMW99}. The parabolic bundle approach applies for any simple $G$. One considers deformations of certain minimally unstable $G$-bundles corresponding to special maximal parabolic subgroups of $G$ (cf. also \cite{And99}). The spectral cover approach applies for $SU(n)$ and $Sp(n)$ bundles and can be essentially understood as a relative Fourier-Mukai transformation. The del Pezzo surface approach applies for $E_6$, $E_7$ and $E_8$ bundles and uses the relation between subgroups of $G$ and singularities of del Pezzo surfaces.

Condition~\eqref{topfive} can be understood as an integrability condition for a generalized anomaly equation which in the context of the present paper is given by
\begin{equation}\label{fivebrane}
i\partial\dbar\omega=\alpha'(\tr R\wedge R-\tr F\wedge F-\sum \delta_5).
\end{equation}
Here the sum is taken over a union of holomorphic curves representing $W$ and $\delta_5$ can be understood as a current that integrates to one in the direction transverse to a single curve. One way to view the $\delta_5$ contribution to the anomaly equation is to consider the gauge field in~\eqref{fivebrane} as a singular limit of smooth solutions of the HYM equation which degenerate such that (a part of) the $\tr F\wedge F$ term becomes a delta function source. This limit can be understood as the analog of the small instanton limit in complex two dimensions which leads to the Uhlenbeck compactification of the underlying moduli space. Another way is to consider the fields in~\eqref{fivebrane} as smooth fields in the non-compact manifold given by the complement of a subvariety on $X$ with behavior at `infinity' fixed by the $\delta_5$ contributions. It is plausible that the implicit function theorem argument can be generalized to find solutions in this last case.

Let us note that solutions of the Strominger system and of the equations of motion describe supersymmetric configurations in the low energy field theory approximation of heterotic string theory. More generally, these equations include terms of higher order in $\alpha'$. In \cite{Strom} it has been argued, heuristically, that solutions of the low energy field theory imply solutions (to all finite orders) of the fully $\alpha'$-corrected equations. This agrees also with predictions one gets from a non-renormalization theorem for the low energy superpotential. Using a perturbative expansion of the gauge connection and the metric around a Calabi-Yau manifold with fixed complex structure~\cite{Wi86} these predictions have been checked ~\cite{WiWi87,WuWi87} (see also~\cite{Strom}) for the first few orders in the expansion in the special case of a deformation of the direct sum of the tangent bundle and the trivial line bundle. The only obstruction encountered in this procedure is the integrability condition imposed by the anomaly equation.
However, it is an open problem to prove that the expansion encounters no obstructions at any order and that exact solutions of the $\alpha'$-corrected equations exist.

Therefore it is plausible that the method applied in this paper leads also to solutions of the
$\alpha'$-corrected equations.

\end{document}